\documentclass[12pt,amscd]{amsart}
\usepackage[all]{xy}
\usepackage{graphicx,tikz}
\usepackage{amsmath,amsxtra,amssymb,latexsym, amscd,amsthm}
\usepackage{indentfirst}
\usepackage[mathscr]{eucal}
\usepackage[pagebackref=true]{hyperref}

\newtheorem{thm}{Theorem}[section]

\newtheorem{lem}[thm]{Lemma}

\theoremstyle{definition}

\newtheorem{exm}[thm]{Example}

\newtheorem{rem}[thm]{Remark}

\newtheorem{conj}{Conjecture}

\numberwithin{equation}{section}

\DeclareMathOperator{\NN}{\mathbb {N}}
\DeclareMathOperator{\ZZ}{\mathbb {Z}}

\DeclareMathOperator{\wdd}{wd}

\DeclareMathOperator{\gr}{gr}

\def\a {\mathbf a}
\def\b {\mathbf b}

\def\v {\mathbf {v}}
\def\w {\mathbf {w}}

\def\mm {\mathfrak m}
\def\AA {\mathbb A}

\begin{document}

\title[Betti numbers of the tangent cones] {Betti numbers of the tangent cones of monomial space curves}
\author[Lan]{Nguyen P. H. Lan}
\address{Vietnam National University, Hanoi University of Science, 334 Nguyen Trai, Thanh Xuan, Hanoi, Vietnam}
\email{lannph@vnu.edu.vn}
\author[Tu]{Nguyen Chanh Tu}
\address{Faculty of Advanced Science and Technology, Danang University of Science and Technology, The University of Danang, 54 Nguyen Luong Bang, Danang, Vietnam}
\email{nctu@dut.udn.vn}

\author[Vu]{Thanh Vu}
\address{Institute of Mathematics, VAST, 18 Hoang Quoc Viet, Hanoi, Vietnam}
\email{vuqthanh@gmail.com}


%\subjclass{13D45, 05C90, 05E40, 05E45.}
%\keywords{Matroid, arboricity, Stanley-Reisner ideal, regularity.}
\date{}

\dedicatory{Dedicated to Professor Ngo Viet Trung on the occasion of his 70th birthday}
\commby{}
%-----------------------------------------------------------
% -----------------------------------------------------------
\maketitle
% -----------------------------------------------------------
\begin{abstract}
    Let $H = \langle n_1, n_2,n_3\rangle$ be a numerical semigroup. Let $\tilde H$ be the interval completion of $H$, namely the semigroup generated by the interval $\langle n_1,n_1+1, \ldots, n_3\rangle$. Let $K$ be a field and $K[H]$ the semigroup ring generated by $H$. Let $I_H^*$ be the defining ideal of the tangent cone of $K[H]$. In this paper, we describe the defining equations of $I_H^*$. From that, we establish the Herzog-Stamate conjecture for monomial space curves stating that $\beta_i(I_H^*) \le \beta_i(I_{\tilde H}^*)$ for all $i$, where $\beta_i(I_H^*)$ and $\beta_i(I_{\tilde H}^*)$ are the $i$th Betti numbers of $I_H^*$ and $I_{\tilde H}^*$ respectively.
\end{abstract}

\maketitle
\section{Introduction}\label{sect_intro} Let $K$ be an arbitrary field. For each sequence $\a = (a_1 < \cdots < a_n)$ of positive integers, denote $C(\a)$ the affine monomial curve in $\AA^n$ parametrized by $(t^{a_1}, \ldots, t^{a_n})$. In \cite{H}, Herzog proved that the number of defining equations of affine monomial space curves is at most three. The situation changes dramatically when we consider monomial curves in higher dimensions. Bresinsky \cite{Br} gave an example of a family of monomial curves in $\AA^4$ whose numbers of defining equations are unbounded. Despite this, Herzog and Srinivasan conjectured, and the third author \cite{V} proved that all the Betti numbers of $C(\a)$ are bounded by a constant depending only on the width, $\wdd(\a) = a_n -a_1$, of the sequence. Unfortunately, the proof of \cite{V} does not reveal an explicit bound on the Betti numbers. In subsequent work, Herzog and Stamate \cite{HS} proved that when $a_1$ is large enough, the Betti numbers of the tangent cone of $C(\a)$ are equal to those of the curve. 

Denote $H = \langle \a \rangle$ the numerical semigroup generated by $\a$ and $\tilde H$ the numerical semigroup generated by the interval completion of $\a$, namely, $\tilde H = \langle a_1, a_1+1, \ldots, a_n\rangle$. Denote $I_H$ the defining ideal of the semigroup ring $K[H] = K[t^{a_1},\ldots, t^{a_n}]$ and $I_H^*$ the defining ideal of the tangent cone of $K[H]$ which is the associated graded ring $\gr_{\mm} K[H]$. Herzog and Stamate conjectured that
\begin{conj}[Herzog-Stamate] Let $H = \langle a_1, a_2, \ldots, a_n \rangle$ be a numerical semigroup. Denote $\tilde H = \langle a_1,a_1+1, \ldots,a_n\rangle$. Then 
$$\beta_i(I_H^*) \le \beta_i(I_{\tilde H}^*),$$
for all $i$, where $\beta_i(I)$ denotes the $i$th Betti number of the ideal $I$.   
\end{conj}
When $\a$ is a generalized arithmetic sequence, namely, $a_i = ua_1 + (i-1)h$ for $i \ge 2$ and some positive integers $u$ and $h$, Sharifan and  Zaare-Nahandi \cite{SZ} described a minimal free resolution of the tangent cone of $C(\a)$. The conjecture is also known for some other special numerical semigroups \cite{Br2, SS}.

As mentioned earlier, the number of defining equations of affine monomial space curves is at most three. Nevertheless, Shibuta \cite[Example 5.5]{GHK} gave a family of monomial space curves whose numbers of defining equations of the tangent cones go to infinity (see \cite{HS} for a treatment of tangent cone of variation of Shibuta semigroups). In this paper, we prove the Herzog-Stamate conjecture for monomial space curves.

\begin{thm}\label{thm_main} Let $H = \langle n_1, n_2,n_3\rangle$ be a numerical semigroup generated by three positive integers $n_1 < n_2 < n_3$. The defining ideal of the tangent cone of $K[H]$ is denoted by $I_H^*$. Let $\wdd(\tilde H) = \min(n_1-1,n_3 - n_1)$ be the width of $\tilde H$. Then 
\begin{enumerate}
    \item $\beta_0(I_H^*) = \mu(I_H^*) \le \wdd(\tilde H) + 1$,
    \item $\beta_1(I_H^*) \le 3 \wdd(\tilde H) - 3,$
    \item $\beta_2(I_H^*) \le 2 \wdd (\tilde H) - 3$.    
\end{enumerate}
In particular, $\beta_i(I_H^*) \le \beta_i(I_{\tilde H}^*)$ for all $i$.
\end{thm}
Note that our bound on the Betti numbers of $I_H^*$ is stronger than the conjectural bound. We now outline the idea of proof of the main theorem. 
\begin{enumerate}
    \item By the result of Herzog \cite{H}, the lattice $\Lambda = \{ \v \in \ZZ^3 \mid \v \cdot \a = 0\}$, where $\a = (n_1,n_2,n_3)$ is generated by two vectors. We first note that $I_H^* = (f_{\v}^* \mid \v \in \Lambda )$. From that, we describe the minimal generators of $I_H^*$ in Section \ref{sec_complete_intersection} and Section \ref{sec_non_ci} depending on whether $I_H$ is a complete intersection or not. 
    \item We show that the minimal generators of $I_H^*$ form a Gr\"obner basis for $I_H^*$ with respect to the reverse lexicographic order. By \cite[Theorem 6.29]{MS}, it suffices to prove the bound for $\mu (I_H^*)$.
    \item From our description of $I_H^*$, we show that $\mu(I_H^*) \le \wdd(\tilde H) +1$.
\end{enumerate}

In \cite{H2, RV}, the authors proved that the Cohen-Macaulay property of the tangent cone in embedding dimension three can be determined from its defining equations. Shen \cite{Sh} further showed that this can be determined from the equations of $I_H$. 

We now describe the organization of the paper. In Section \ref{sec_pre}, we set up the notation, recall the result of Herzog \cite{H}, and prove that $I_H^*$ is generated by the initial forms of binomials in $I_H$. In Section \ref{sec_complete_intersection}, we describe $I_H^*$ when $I_H$ is a complete intersection. In Section \ref{sec_non_ci}, we describe $I_H^*$ when $I_H$ is not a complete intersection. In Section \ref{sec_betti_bounds}, we prove our main theorem.

\section{Preliminaries}\label{sec_pre}

Let $H = \langle a_1, \ldots, a_n\rangle$ and $S = K[x_1, \ldots, x_n]$ the polynomial in $n$ variables over a field $K$. Let $\Lambda \subset \ZZ^n$ be defined by $\Lambda = \{ \v \in \ZZ^n \mid \v \cdot \a = 0\}$, where $\a = (a_1, \ldots, a_n)$. For each $\b = (b_1, \ldots, b_n) \in \NN^n$, we denote by $x^\b$ the monomial $x_1^{b_1} \cdots x_n^{b_n}$. For each vector $\v \in \ZZ^n$, denote $f_\v$ the binomial $x^{\v_+} - x^{\v_-}$ where $\v_+$ and $\v_-$ are defined by 
\begin{align*}
    (\v_+)_i &= \max (\v_i,0)\\
    (\v_-)_i &=\max (-\v_i,0).
\end{align*}
For a polynomial $f$ in $S$, $f^*$ denotes the initial form of $f$, namely the homogeneous component of the smallest degree of $f$. First, we have
\begin{lem}\label{lem_computation_tangent_cone} We have 
    $$I_H^* = (f_\v^* \mid \v \in \Lambda).$$ 
\end{lem}
\begin{proof}
    By definition we have $I_H^* = (f^*|f\in I_H)$. By \cite[Theorem 3.2]{HHO}, $I_H = (f_\v \mid \v \in \Lambda)$. By \cite[Theorem 15.28]{E}, $I_H^* = (f^* \mid f \text{ in a Gr\"obner basis of } I_H)$. By \cite[Theorem 3.6]{HHO}, $I_H$ has a Gr\"obner basis consisting of binomials of the form $f_\v$ with $\v \in \Lambda$. The conclusion follows.
\end{proof}
Now, we come back to our situation of monomial space curves. We may assume that $n_1, n_2,n_3$ minimally generate the numerical semigroup $H = \langle n_1, n_2, n_3\rangle$. In particular, $\gcd(n_1,n_2,n_3) = 1$. Denote $a = n_2 - n_1$, $b = n_3 - n_1$, and $d = \gcd(a,b)$.

Since we are working with monomial space curves, to avoid indexing, we use the variables $x,y,z$ instead of $x_1,x_2,x_3$. Throughout the paper, we also use the following terminology from Herzog \cite{H}. Let $\v \in \ZZ^3$ be a vector. We say that 
\begin{enumerate}
\item $\v$ and $f_\v$ is of type $1$ if $f_\v$ is of the form $\pm (x^{a_1} - y^{a_2} z^{a_3})$. In this case, $f_\v^* = y^{a_2}z^{a_3}$.
\item $\v$ and $f_\v$ is of type $2$ if $f_\v$ is of the form $\pm (y^{a_2} - x^{a_1}z^{a_3})$.
\item $\v$ and $f_\v$ is of type $3$ if $f_\v$ is of the form $\pm (z^{a_3} - x^{a_1}y^{a_2})$. In this case, $f_\v^* = z^{a_3}$.
\end{enumerate}
When it is clear from the context, we denote by $I$ the ideal $I_H$. Let 
$$J = (f \mid f \text{ is homogeneous in } I)$$ and call it the homogeneous part of $I$. The ideal $J$ plays an important role in studying the Betti numbers of the shifted family $I(\a +k)$ \cite{V}. By \cite{S2}, it also has a nice interpretation for the complete intersection property of the ideals in the shifted family $I(\a + k)$.

\begin{lem}\label{lem_homogeneous} Let $H = \langle n_1, n_2, n_3\rangle$. Let $J$ be the homogeneous part of $I_H$. Then $J = (y^{b/d} - x^{(b-a)/d} z^{a/d})$, where $a = n_2 - n_1, b = n_3 - n_1$ and $d = \gcd(a,b)$.    
\end{lem}
\begin{proof} Assume that $y^{u+v} - x^u z^v$ is a homogeneous binomial in $I_H$. Then we have 
$$(u+v) (n_1 + a) = u n_1 + v(n_1+b).$$
Hence, $u a = v(b-a)$. The conclusion follows.    
\end{proof}

We now recall the result of Herzog \cite{H} describing the generators of the lattice $\Lambda$ associated with $H$. Let $c_1, c_2,c_3$ be the smallest positive integers such that $c_1 n_1 \in \langle n_2, n_3 \rangle$, $c_2 n_2 \in \langle n_1, n_3 \rangle$ and $c_3 n_3 \in \langle n_1, n_2 \rangle$. Hence, there exist non-negative integers $r_{ij}$ satisfying the equations
\begin{align}
    c_1 n_1 &= r_{12} n_2 + r_{13} n_3\\
    c_2 n_2 & = r_{21} n_1 + r_{23} n_3 \\
    c_3 n_3 & = r_{31} n_1 + r_{32} n_2.
\end{align}
Denote $\v_1 = (c_1,-r_{12},-r_{13})$, $\v_2 = (-r_{21},c_2,-r_{23})$ and $\v_3 = (-r_{31},-r_{32},c_3)$. Herzog proved that if $r_{ij} = 0$ for some $i,j$, then two of the vectors are negated of each other, and $I_H$ is a complete intersection. If $r_{ij} \neq 0 $ for all $i,j$, then $\v_1 + \v_2 + \v_3 = 0$. In either case, $\Lambda$ is generated by two of the three vectors, called $\w_1$ and $\w_2$. Hence,
\begin{lem}\label{lem_tangent_cone_space_curve} Let $H = \langle n_1, n_2, n_3 \rangle$ and $\Lambda$ be the associated lattice which is generated by $\w_1$ and $\w_2$. Then 
$$ I_H^* = (f_{\w_1}^*, f_{\w_2}^*, f_{a_1 \w_1 + a_2 \w_2}^*, f_{b_1 \w_1 - b_2 \w_2}^* \mid a_1, a_2, b_1,b_2 > 0, (a_1,a_2) = 1 \text{ and } (b_1,b_2) = 1).$$
\end{lem}
\begin{proof}
    Follows from Lemma \ref{lem_computation_tangent_cone} and the fact that $f_\v = -f_{-\v}$ and $f_{k\v}^* \in (f_\v^*)$.
\end{proof}
Finally, we introduce the following notation that arises frequently in the sequence. For a real number $r$, denote $\lceil r \rceil$ the least integer at least $r$, $\lfloor r \rfloor$ the largest integer at most $r$. We also define a function $\varphi_r: \NN \to \NN$ as follows. $$\varphi_r( t)  = \begin{cases} rt + 1 & \text{ if } rt \in \NN, \\
\lceil rt \rceil &\text{ otherwise}.\end{cases}$$
Let $\frac{a}{b} < \frac{c}{d}$ be two positive rational numbers. Let $f$ be the smallest positive integer such that there exists an $h$ satisfying $\frac{a}{b} < \frac{h}{f} \le \frac{c}{d}$. We then let $e$ be the smallest such $h$. We denote $\frac{e}{f}$ by $M(\frac{a}{b},\frac{c}{d})$. Note that when $f$ is determined then $e = \varphi_r (f)$, where $r = a/b$. The fraction $\frac{e}{f}$ can be found from the continued fractions of $\frac{a}{b}$ and $\frac{c}{d}$ and is tightly related to the Farey sequence.

\section{Complete intersection monomial space  curves}\label{sec_complete_intersection}
We keep the notation as in Section \ref{sec_pre}. In this section, we describe $I_H^*$ when $I_H$ is a complete intersection. There are three possibility, namely, either $r_{12} = r_{32} = 0$, $r_{13} = r_{23} = 0$, or $r_{21} = r_{31} = 0$.

\begin{lem}\label{lem_ci_1}
    Assume that $r_{13} = 0$. Then $I^* = (y^{c_2},z^{c_3})$.
\end{lem}
\begin{proof}
Since $r_{13} = 0$, we know that the lattice $\Lambda$ is generated by $\w_1 = (c_1,-c_2,0)$ and $\w_2 = (r_{31},r_{32},-c_3)$. In particular, $y^{c_2}, z^{c_3} \in I^*$. Let $L = (y^{c_2},z^{c_3})$. By Lemma \ref{lem_tangent_cone_space_curve}, it suffices to prove that for any $a_1, a_2 > 0$ and all $b_1, b_2 > 0$, $f_{\w}^* \in L$, where $\w = a_1 \w_1 + a_2 \w_2$ or $\w = b_1 \w_1 - b_2 \w_2$. Since, $a_2,b_2 > 0$, $w_3$ is a multiple of $c_3$. Hence, if $\w$ is of type $1$ or type $3$ then $f_\v^* \in (z^{c_3}) \subset L$. Now assume that $\w$ is of type $2$. There are three cases.

\noindent \textbf{Case 1.}  $f_\w^* = x^{w_1}z^{w_3}$, in this case, $f_\w^* \in (z^{c_3})$.

\noindent \textbf{Case 2.}  $f_\w^* = y^{w_2}$. Since $c_2$ is minimal such that $c_2n_2\in \langle n_1,n_3\rangle$, we have $w_2 \ge c_2$. In particular, $f_\w^* \in (y^{c_2})$.

\noindent \textbf{Case 3.}  $f_\w$ is homogeneous. With the same reasoning as in Case 2, we have $w_2 \ge c_2$. Furthermore, $w_3$ is a multiple of $c_3$, both components of $f$ belong to $L$, hence $f_\w\in L$.
\end{proof}

\begin{exm} Let $H = \langle 20,30,37\rangle$. Then $I_H = (x^3 - y^2, x^2 y^{11} - z^{10})$ and $I_H^* = (y^2,z^{10})$.     
\end{exm}

\begin{rem} The defining ideal of the tangent cone is the stabilized partial elimination ideal of the projectivization of $I_H$. That is how we compute it in Macaulay2 \cite{M2}. See \cite{KNV} for more information about the partial elimination ideals and their depth and regularity in comparison to the depth and regularity of the projectivization of $I_H$.
\end{rem}

Now, consider the case $r_{21} = r_{31} = 0$. By Lemma \ref{lem_ci_1}, we may assume that $r_{13} > 0$. Thus, the lattice $\Lambda$ is generated by $\w_1 = (c_1, -r_{12}, -r_{13})$ and $\w_2 = (0,c_2,-c_3)$. If $r_{13} \ge c_3$, then we can replace $r_{12}$ by $r_{12} + c_2$ and $r_{13}$ by $r_{13} - c_3$. Thus, we may assume that $0 < r_{13} < c_3$. We fix the following notation. Let $\delta_1 = c_1 - (r_{12} + r_{13})$ and $\delta_2 = c_2 - c_3$. Since $\w = (-(b-a)/d,b/d,-a/d) \in \Lambda$, there exist $\alpha$ and $\beta$ such that $\w = \alpha \w_1 + \beta \w_2$. In particular, $\alpha c_1 = -(b-a)/d$ and $\alpha \delta_1 + \beta \delta_2 = 0$. Denote $\alpha_h = -\alpha$ and $\beta_h = \beta$. Then $\alpha_h,\beta_h > 0$ and 
\begin{equation}\label{eq_ci_2_1}
    \begin{cases}
    \alpha_h c_1 &= (b-a)/d\\
    \alpha_h r_{12} + \beta_h c_2 &= b/d \\
    -\alpha_h r_{13} + \beta_h c_{3} & = a/d.
    \end{cases}
\end{equation}
First, we claim 
\begin{lem}\label{lem_ci_2_1} Assume that $r_{21} = 0$. Then $\frac{\delta_2}{\delta_1} < \frac{c_3}{r_{13}}$.    
\end{lem}
\begin{proof} By \eqref{eq_ci_2_1}, $-\alpha_h r_{13} + \beta_h c_3 = a/d > 0$. Hence, 
$$\frac{\delta_2}{\delta_1} = \frac{\alpha_h}{\beta_h} < \frac{c_3}{r_{13}}.$$
The conclusion follows.    
\end{proof}

Let $\frac{\eta}{\epsilon} = M(\frac{\delta_2}{\delta_1},\frac{c_3}{r_{13}})$, i.e., $\eta$ and $\epsilon$ are smallest positive integers such that $\frac{\delta_2}{\delta_1} < \frac{\eta}{\epsilon} \le \frac{c_3}{r_{13}}$. We now define the following two sets. 

Let $\alpha_0 = 1, \beta_0 = 0$. For each $i \ge 0$, consider the system of inequalities 
\begin{equation} \label{eq_ci_2_2}
\begin{cases}
     \alpha r_{13} - \beta c_3 & > 0 \\
    \alpha r_{13} - \beta c_3 & < \alpha_i r_{13} - \beta_i c_3 \\
    \beta &> \beta_i.
\end{cases}
\end{equation}
When the system \eqref{eq_ci_2_2} has solutions, let $\beta_{i+1}$ be the smallest solution among $\beta$. Then $\alpha_{i+1} = \varphi_r (\beta_{i+1})$ where $r = c_3 / r_{13}$. Note that there can be only finitely many such $\beta_i$ as the sequence $\{\alpha_i r_{13} - \beta_i c_3\}$ is a strictly decreasing sequence of positive integers. Let $t$ be the largest index of $i$ such that $\beta_t < \epsilon$, and $A = \{ (\alpha_i,\beta_i) \mid i = 1, \ldots, t\}$. Since $\alpha_0 r_{13} - \beta_0 c_3 = r_{13}$, $|A| \le r_{13}$. For each $\alpha,\beta \in \NN$ such that $\alpha r_{13} -\beta c_3 \ge 0$, let $p_{\alpha,\beta} = y^{\alpha r_{12} +\beta c_2} z^{\alpha r_{13} - \beta c_3}$. We have

\begin{lem}\label{lem_ci_2_2} Assume that $\alpha' > 0$ and $\beta' \ge 0$ be such that $\alpha' r_{13} - \beta' c_3 \ge 0$. Let $g = y^{\eta r_{12} + \epsilon c_2}$. Then $p_{\alpha',\beta'} \in L = (g,p_{\alpha,\beta} \mid (\alpha,\beta) \in A)$.    
\end{lem}
\begin{proof} If $\beta' = 0$ then $p_{\alpha',\beta'}$ is divisible by $p_{1,0}$. Assume that $\beta' \ge \epsilon$. Since $\frac{\alpha'}{\beta'} \ge \frac{c_3}{r_{13}} \ge \frac{\eta}{\epsilon}$, we deduce that $\alpha' \ge \eta$. Thus $p_{\alpha',\beta'}$ is divisible by $g$.
    Now, assume that $0 < \beta' < \epsilon$. Let $i$ be the largest index such that $\beta_i < \beta'$. Let $r = \frac{c_3}{r_{13}}$. Since $\frac{\alpha'}{\beta'} \ge r$ and $\alpha_i = \varphi_r(\beta_i)$ we deduce that $\alpha' \ge \alpha_i$. If $\alpha' r_{13} - \beta' c_3 \ge \alpha_i r_{13} - \beta_i r_{23}$ then $p_{\alpha',\beta'}$ is divisible by $p_{\alpha_i,\beta_i}$. If $\alpha' r_{13} - \beta' c_3 < \alpha_i r_{13} - \beta_i r_{23}$ then $(\alpha',\beta')$ is a solution to the system \eqref{eq_ci_2_2}. By our assumption, we must have $\beta' = \beta_{i+1}$. Hence $\alpha' \ge \alpha_{i+1}$ and $p_{\alpha',\beta'}$ is divisible by $p_{\alpha_{i+1},\beta_{i+1}}$. The conclusion follows.
\end{proof}

Let $\gamma_0 = 0, \sigma_0 = 1$. For each $i \ge 0$, consider the system of inequalities 
\begin{equation} \label{eq_ci_2_3}
\begin{cases}
     -\gamma \delta_1 + \sigma \delta_2 & > 0 \\
    -\gamma r_{13} + \sigma c_3 & < -\gamma_i r_{13} + \sigma_i c_3 \\
    \gamma &> \gamma_i.
\end{cases}
\end{equation}
When the system \eqref{eq_ci_2_3} has solutions, let $\gamma_{i+1}$ be the smallest solution among $\gamma$. Then $\sigma_{i+1} = \varphi_r (\gamma_{i+1})$ where $r = \delta_1 / \delta_2$. There can be only finitely many such $\gamma_i$ as the sequence $\{-\gamma_i r_{13} + \sigma_i c_3\}$ is a strictly decreasing sequence of positive integers. Let $u$ be the largest index of $i$ and $B = \{ (\gamma_i, \sigma_i) \mid i = 0, \ldots, u\}$. Since $-\gamma_0 r_{13} + \sigma_0 c_3 = c_3$, $|B| \le c_3$. For each $\gamma,\sigma \in \NN$ such that $-\gamma \delta_1 + \sigma \delta_2 > 0$, let $q_{\gamma,\sigma} = x^{\gamma c_1} z^{-\gamma r_{13}  + \sigma c_3}$. We have
   
\begin{lem}\label{lem_ci_2_3} Assume that $\gamma',\sigma' \in \NN$ be such that $-\gamma' \delta_1 + \sigma' \delta_2 > 0$. Then $q_{\gamma',\sigma'} \in L = (q_{\gamma,\sigma} \mid (\gamma,\sigma) \in B)$.    
\end{lem}
\begin{proof} If $\gamma' = 0$ then $q_{\gamma',\sigma'}$ is divisible by $q_{0,1}$. Thus, we may assume that $\gamma' > 0$. Let $i$ be the largest index such that $\gamma_i < \gamma'$. Since $\sigma_i = \varphi_r(\gamma_i)$ and $\frac{\sigma'}{\gamma'} > r$ where $r = \delta_1/\delta_2$ we deduce that $\sigma' \ge \sigma_i$. If $-\gamma' r_{13} + \sigma' c_3 \ge -\gamma_i r_{13} + \sigma_i c_{3}$ then $q_{\gamma',\sigma'}$ is divisible by $q_{\gamma_i,\sigma_i}$. If $-\gamma' r_{13} + \sigma' c_3 < -\gamma_i r_{13} + \sigma_i c_{3}$ then $(\gamma',\sigma')$ is a solution to the system \eqref{eq_ci_2_3}. By our assumption, we must have $\gamma' = \gamma_{i+1}$. Hence $\sigma' \ge \sigma_{i+1}$ and $q_{\gamma',\sigma'}$ is divisible by $q_{\gamma_{i+1},\sigma_{i+1}}$. The conclusion follows.
\end{proof}

\begin{lem}\label{lem_ci_2}
    Assume that $r_{21} = 0$. Let $r_{12},r_{13}$ be such that $c_1 n_1 = r_{12} n_2 + r_{13} n_3$ with $0 < r_{13} < c_3$. Let $g = y^{\eta r_{12} + \epsilon c_2}$ and $h = y^{b/d} - x^{(b-a)/d}z^{a/d}$. Then 
       \begin{equation}
           I^* = (g, h) + ( p_{\alpha,\beta} \mid (\alpha,\beta) \in A) 
           + (q_{\gamma,\sigma} \mid (\gamma,\sigma) \in B).
       \end{equation}
       Furthermore, 
       \begin{enumerate}
           \item If $\beta_h > \epsilon$ then $I^* = (g) + (p_{\alpha,\beta} \mid (\alpha,\beta) \in A) + (q_{\gamma,\sigma} \mid (\gamma,\sigma) \in B)$.
           \item If $\beta_h \le \epsilon$ then $I^* = (h) + (p_{\alpha,\beta} \mid (\alpha,\beta) \in A, \beta < \beta_h) + (q_{\gamma,\sigma} \mid (\gamma,\sigma) \in B)$.
           \item The minimal generators of $I^*$ form a Gr\"obner basis of $I^*$ with respect to the reverse lexicographic order.
           \item $\mu(I^*) \le c_3+2$.
       \end{enumerate}
\end{lem}
\begin{proof} Recall that $\w_1 = (c_1,-r_{12}, - r_{13})$ and $\w_2 = (0,c_2,-c_3)$. Let 
$$ L = (g, h) + ( p_{\alpha,\beta} \mid (\alpha,\beta) \in A) + (q_{\alpha,\beta} \mid (\alpha, \beta) \in B).$$
For clarity, we divide the proof into several steps.

\vspace{2mm}
\noindent \textbf{Step 1.} $L \subseteq I^*$. Indeed, by the choice of $\eta$ and $\epsilon$ we have $-\eta r_{13} + \epsilon c_3 \ge 0$ and $\eta c_1 + (-\eta r_{13} + \epsilon c_3) > \eta r_{12} + \epsilon c_2$. Hence, $f_{\eta \w_1 - \epsilon \w_2}^* = y^{\eta r_{12} + \epsilon c_2} = g \in I^*$. From the definition of $A$ and $B$, we see the following. For $(\alpha,\beta) \in A$, $f_{\alpha \w_1 - \beta \w_2}^* = y^{\alpha r_{12} + \beta c_2} z^{\alpha r_{13} - \beta c_3} = p_{\alpha,\beta}$. For $(\gamma, \sigma) \in B$, $f_{\gamma \w_1 - \sigma \w_2}^* = x^{\gamma c_1} z^{-\gamma r_{13} + \sigma c_3} = q_{\gamma,\sigma}$. 

\vspace{2mm}
\noindent \textbf{Step 2.} $I^* \subseteq L$. First, note that $(1,0) \in A$ and $(0,1) \in B$ and we have $y^{r_{12}}z^{r_{13}}, z^{c_3} \in I^*$. Now, we prove that for any $\alpha, \beta > 0$, $f_{\alpha \w_1 + \beta \w_2}^* \in (z^{c_3})$. Indeed, we have $\alpha \w_1 + \beta \w_2 = (\alpha c_1, -\alpha r_{12} + \beta c_2, -\alpha r_{13} - \beta c_3)$. Thus, it is either of type $1$ or type $3$, hence $f_{\alpha \w_1 + \beta \w_2}^*$ is divisible by $z^{\alpha r_{13} + \beta c_3}$. 

By Lemma \ref{lem_tangent_cone_space_curve}, it suffices to prove that for any $\alpha, \beta > 0$, $f_\w^* \in L$ where $\w = \alpha \w_1 - \beta \w_2 = (\alpha c_1, -\alpha r_{12} - \beta c_2, -\alpha r_{13} + \beta c_3)$. There are two cases. 

\noindent \textbf{Case 1.} $-\alpha r_{13} + \beta c_3 < 0$. We have $f_\w^* = p_{\alpha,\beta} \in L$ by Lemma \ref{lem_ci_2_2}.

\vspace{1.5mm}

\noindent \textbf{Case 2.} $-\alpha r_{13} + \beta c_3 \ge 0$. There are three subcases. 

Case 2.a. $\alpha \delta_1 > \beta \delta_2$. In particular, $\frac{\delta_2}{\delta_1} < \frac{\alpha}{\beta} \le \frac{c_3}{r_{13}}$. By definition of $\eta, \epsilon$ we have $\beta \ge \epsilon$ and $\alpha \ge \eta$. Hence, $f_\w^* \in (g)$. 

Case 2.b. $\alpha \delta_1 < \beta \delta_2$. Then $f_\w^* = x^{\alpha c_1} z^{-\alpha r_{13} + \beta c_3} = q_{\alpha,\beta} \in L$ by Lemma \ref{lem_ci_2_3}.

Case 2.c. $\alpha \delta_1 = \beta \delta_2$. Then $f_{\w}^* \in (h)$. 

\vspace{2mm}
\noindent \textbf{Step 3.} Assume that $\beta_h > \epsilon$. Then $h \in L = (g) + (p_{\alpha,\beta} \mid (\alpha,\beta) \in A) + (q_{\gamma,\sigma} \mid (\gamma,\sigma) \in B)$. Since $\beta_h > \epsilon$ and $\eta = \varphi_r(\epsilon)$ where $r = \delta_2/\delta_1 = \alpha_h/\beta_h$, we deduce that $\alpha_h \ge \eta$. Thus, it suffices to prove that $q_{\alpha_h,\beta_h} \in L$. Let $\beta_h = k \epsilon + \beta$ for $0 \le \beta < \beta_h$. Then we also have $\alpha_h \ge k \eta$. If $\beta = 0$ then $k > 1$; let $\gamma = \alpha_h - (k-1) \eta$ and $\sigma = \beta_h - (k-1) \epsilon$. If $\beta > 0$, let $\gamma = \alpha_h - k\eta$ and $\sigma = \beta_h - k \epsilon$. Then $\alpha_h / \beta_h > \gamma /\sigma$. Furthermore, 
$$(-\alpha_h r_{13} + \beta_h c_3) - (-\gamma r_{13} + \sigma c_3) = \ell (\epsilon c_3 - \eta r_{13}) \ge 0,$$
where $\ell = k-1$ if $\beta = 0$ and $k$ if $\beta > 0$. Thus, $q_{\alpha_h,\beta_h}$ is divisible by $q_{\gamma,\sigma}$ which belongs to $L$ by Lemma \ref{lem_ci_2_3}. 

\vspace{2mm}
\noindent \textbf{Step 4.} Assume that $\beta_h \le \epsilon$. Let $L = (h) + (p_{\alpha,\beta} \mid (\alpha,\beta) \in A, \beta < \beta_h) + (q_{\gamma,\sigma} \mid (\gamma,\sigma) \in B)$. Then $g$ and $p_{\alpha,\beta}$ with $(\alpha,\beta)\in A$, $\beta \ge \beta_h$ belongs to $L$.

For $g$: Let $\epsilon = k \beta_h + \beta$ where $0 \le \beta < \beta_h$. First, assume that $\beta = 0$. Since $\eta = \varphi_r(\epsilon)$ where $r = \alpha_h/\beta_h$, we deduce that $\eta = k \alpha_h + 1$. Thus,
$$ \epsilon c_{3} = k \beta_h c_3 \ge \eta r_{13} = (k \alpha_h + 1) r_{13}.$$ 
Hence,
$$g - y^{(\eta - k \alpha_h)r_{12}} f_{k(-\alpha_h \w_1 +\beta_h \w_2)} = y^{(\eta - k \alpha_h)r_{12}} x^{k\alpha_h c_1} z^{-k\alpha_h r_{13} + k \beta_h c_3} \in (y^{r_{12}} z^{r_{13}}).$$

Now, assume that $\beta > 0$, let $\alpha = \eta - k \alpha_h$. Then $\alpha/\beta > \alpha_h/\beta_h$. By the choice of $\eta$ and $\epsilon$ we must have $\alpha/\beta > c_3 / r_{13}$. Thus,
$$g - y^{\alpha r_{12} + \beta c_2} f_{k (-\alpha_h \w_1 + \beta_h \w_2)} = x^{k \alpha_h c_1} y^{\alpha r_{12} + \beta c_2} z^{-k \alpha_h r_{13} + k \beta_h c_3}.$$
Furthermore, 
$$(-k \alpha_h r_{13} + k \beta_h c_3) - (\alpha r_{13} - \beta c_3) = \epsilon c_3 - \eta r_{13} \ge 0.$$
Thus $g - y^{\alpha r_{12} + \beta c_2} f_{k (-\alpha_h \w_1 + \beta_h \w_2)}$ is divisible by $p_{\alpha,\beta}$ which belongs to $L$ by Lemma \ref{lem_ci_2_2}.

For $p_{\alpha,\beta}$ with $(\alpha,\beta) \in A$ and $\beta \ge \beta_h$. We proceed similarly to the previous case. Let $\beta = k \beta_h + \beta'$. Since $\frac{\alpha}{\beta} > \frac{\alpha_h}{\beta_h}$, $\alpha > k \alpha_h$. Let $\alpha' = \alpha - k \alpha_h$. There are two cases.  

\noindent \textbf{Case 1.} $\beta' = 0$. We have 
$$p_{\alpha,\beta} - y^{\alpha' r_{12}}z^{\alpha r_{13} - \beta c_3} f_{k (-\alpha_h \w_1 + \beta_h \w_2)} = x^{k \alpha_h c_1} y^{\alpha'r_{12}} z^{\alpha' r_{13}} \in (y^{r_{12}} z^{r_{13}}).$$

\noindent \textbf{Case 2.} $\beta' > 0$. Then $\alpha'/\beta' > \alpha_h/\beta_h$. By the definition of $\epsilon$, we must have $\alpha'/\beta' > c_3/r_{13}.$ Hence, $$p_{\alpha,\beta} - y^{\alpha' r_{12} + \beta' c_2} z^{\alpha r_{13} - \beta c_3}f_{k (-\alpha_h \w_1 + \beta_h \w_2)} = x^{k \alpha_h c_1} y^{\alpha' r_{12} + \beta' c_2} z^{\alpha' r_{13} - \beta' c_3}$$
is divisible by $p_{\alpha',\beta'}$ which belongs to $L$ by Lemma \ref{lem_ci_2_2}.

\vspace{2mm}
\noindent \textbf{Step 5.} Assume that $\beta_h \le \epsilon$. Then $\{h\} \cup \{p_{\alpha,\beta} \mid (\alpha,\beta) \in A, \beta < \beta_h \} \cup \{q_{\gamma,\sigma} \mid (\gamma,\sigma) \in B\}$ form a Gr\"obner basis of $I^*$ with respect to the reverse lexicographic order. 

By Buchberger's criterion \cite[Theorem 15.8]{E}, it suffices to prove that for all $(\alpha, \beta) \in A$ with $\beta < \beta_h$, the $S$-pair of $h$ and $p_{\alpha,\beta}$ reduces to $0$. First, we claim that $\alpha \le \alpha_h$. Assume that $\alpha > \alpha_h$. Since $\alpha = \varphi_r(\beta)$ where $r = \frac{c_3}{r_{13}}$, we deduce that 
$$\frac{\alpha_h}{\beta_h} < \frac{\alpha_h}{ \beta} \le \frac{\alpha - 1}{\beta} < \frac{c_3}{r_{13}}.$$
Thus, $\beta \ge \epsilon$. But that is a contradiction as $\beta < \beta_h \le \epsilon$.

We have 
\begin{equation}
    S (p_{\alpha,\beta}, h) = x^{\alpha_h c_1 } z^{-(\alpha_h - \alpha) r_{13} + (\beta_h -  \beta) c_3}. 
\end{equation}
In particular, $S(p_{\alpha,\beta},h)$ is divisible by $q_{\alpha_h - \alpha,\beta_h-\beta}$. Furthermore, $(\alpha_h - \alpha)/(\beta_h - \beta) < \alpha_h/\beta_h$; thus, $q_{\alpha_h - \alpha, \beta_h -b}$ is divisible by $q_{\gamma_i,\sigma_i}$ for some $(\gamma_i,\sigma_i) \in B$ by Lemma \ref{lem_ci_2_3}. The conclusion follows.

\vspace{2mm}
\noindent \textbf{Step 6.} $\mu(I^*) \le c_3+2$. Let $c_3 = k r_{13} + s_{13}$. If $s_{13} = 0$, then $\alpha r_{13} - \beta c_3$ is divisible by $r_{13}$. Thus, the system \eqref{eq_ci_2_2} has no solution for $i \ge 0$. In other words, $|A| = 1$. Furthermore, by Step 3 and Step 4, $h$ and $g$ cannot be both minimal. Hence, $\mu(I^*) \le 1 + 1 + c_3= c_3 + 2$. 

Now assume that $s_{13} > 0$. If $\frac{\delta_2}{\delta_1} \le 1$, then the system \eqref{eq_ci_2_3} has no solutions. In particular, $|B| = 1$. Thus $\mu(I^*) \le 1 + |A| + 1 \le 1 + r_{13} + 1 \le c_3 + 1$. If $\frac{\delta_2}{\delta_1} > 1$, then $(1,1) \in B$ and $q_{1,1} = x^{c_1} z^{c_1 - r_{13}}$. Hence, $|B| \le c_1 - r_{13} + 1$. Furthermore, $(k+1,1) \in A$, and $p_{k+1,1} = y^{(k+1)r_{12} + c_2} z^{r_{13} - s_{13}}$. Hence $|A| \le r_{13} - s_{13} + 1$. Thus, 
$$ \mu(I^*) \le 1 + |A| + |B| \le 1 + (c_1 - r_{13} + 1)  + (r_{13} -  s_{13} +1) \le c_3 + 2.$$
That concludes the proof of the lemma.
\end{proof}

\begin{exm} Let $H = \langle 332,345,450\rangle$. Then $I_H = (x^{15} - y^4z^8,y^{30}-z^{23})$. We have $\delta_1 = 3$, $\delta_2 = 7$, $c_3 = 23$, $r_{13} = 8$. The smallest $\eta,\epsilon$ such that $\frac{\delta_2}{\delta_1} < \frac{\eta}{\epsilon} \le \frac{c_3}{r_{13}}$ are $5,2$. We have $I_H^* = (z^{23},y^4z^8,y^{42} z, y^{80},x^{15}z^{15},x^{30}z^7)$.  
\end{exm}

\begin{lem}\label{lem_ci_3}
    Assume that $r_{12} = 0$ and $c_2 < r_{21} + r_{23}$. Then $I^* = (y^{c_2},z^{c_3})$.
\end{lem}
\begin{proof}
    Clearly, by assumption, $y^{c_2}, z^{c_3} \in I^*$. The lattice $\Lambda$ is generated by $\w_1 = (c_1,0,-c_3)$ and $\w_2 = (-r_{21},c_2,-r_{23})$. By Lemma \ref{lem_computation_tangent_cone}, it suffices to prove that $f_\w^* \in J = (y^{c_2},z^{c_3})$ for all $\w \in \Lambda$. By Lemma \ref{lem_tangent_cone_space_curve}, there are two cases. 
    
 \noindent \textbf{Case 1.} $\w = \alpha \w_1 + \beta \w_2 = (\alpha c_1 - \beta r_{21}, \beta c_2, - \alpha c_3 - \beta r_{23})$. Since $y^{\beta c_2}, z^{\alpha c_3 + \beta r_{23}} \in J$, $f_\w^* \in J$. 

\noindent \textbf{Case 2.} $\w = \alpha \w_1 - \beta \w_2 = (\alpha c_1 + \beta r_{21}, -\beta c_2, -\alpha c_3 + \beta r_{23})$. If $-\alpha c_3 + \beta r_{23} < 0$, then $f_\w^*$ is divisible by $y^{c_2}$. If $-\alpha c_3 + \beta r_{23} > 0$ then $\alpha c_1 + \beta r_{21} - \alpha c_3 + \beta r_{23} > \beta c_2$, hence $f_\w^* = y^{\beta c_2} \in J$. The conclusion follows.
\end{proof}
\begin{exm}Let $H = \langle 20,23,30\rangle$. Then $I_H = (x^3 - z^2, y^{10} - x^{10}z)$. Hence, $I_H^* = (y^{10}, z^2).$    
\end{exm}

\begin{lem}\label{lem_ci_4} Assume that $r_{12} = 0$ and $c_2 = r_{21} + r_{23}$. Then $I^* = (y^{c_2} -x^{r_{21}} z^{r_{23}}, z^{c_3})$.    
\end{lem}
\begin{proof}The lattice $\Lambda$ is generated by $\w_1 = (c_1,0,-c_3)$ and $\w_2 = (-r_{21}, c_2, -r_{23})$. First, we have $f_{\w_1}^*$ and $f_{\w_2} = f_{\w_2}^* \in I^*$. Let $J = (y^{c_2} -x^{r_{21}} z^{r_{23}}, z^{c_3}).$ We will now prove that $I^* \subseteq J$. By Lemma \ref{lem_tangent_cone_space_curve}, there are two cases.

\noindent \textbf{Case 1.} $\w = \alpha \w_1 + \beta \w_2 = (\alpha c_1 - \beta r_{21}, \beta c_2, -\alpha c_3 - \beta r_{23})$. There are two subcases. 

Case 1.a. $\alpha c_1 - \beta r_{21} < 0$. Then we have $\beta r_{21} - \alpha c_1 + \alpha c_3 + \beta r_{23} < \beta c_2$. Hence, $f_\w^* = x^{\beta r_{21} - \alpha c_1} z^{\alpha c_3 + \beta r_{23}}  \in (z^{c_3})$. 

Case 1.b. $\alpha c_1 - \beta r_{21} > 0$. Then $f^* = z^{\alpha c_3 + \beta r_{23}} \in (z^{c_3})$.

\noindent \textbf{Case 2.} $\w = \alpha \w_1 - \beta \w_2 = (\alpha c_1 + \beta r_{21}, - \beta c_2, -\alpha c_3 + \beta r_{23}).$ There are two subcases. 

Case 2.a. $-\alpha c_3 + \beta r_{23} \ge  0$. Since $c_2 = r_{21} + r_{23}$, we have $f_\w^* = y^{\beta c_2}$. Let $c_3 = k r_{23} + s_{23}$ with $0 \le s_{23} < r_{23}$. We have 
$$y^{\beta c_2} - f_{\beta \w_2} = x^{\beta r_{21}} z^{\beta r_{23}}.$$
Since $\beta r_{23} \ge \alpha c_3 \ge c_3$, we deduce that $f_w^* \in J$.

Case 2.b. $-\alpha c_3 + \beta r_{23} < 0$. Then $f_\w^* = y^{\beta c_2} z^{\alpha c_3 - \beta r_{23}}$. We have
$$f_\w^* - f_{\beta \w_2} = x^{\beta r_{21}} z^{\alpha c_3} \in (z^{c_3}).$$
The conclusion follows.    
\end{proof}
\begin{exm}
    This is the typical case of shifted semigroup which is a complete intersection with large enough shift. For example, consider the example in the previous case shifted by $140$, namely $H = \langle 160,163,170 \rangle$. Then $I = (x^{17} - z^{16}, y^{10} - x^7z^3)$ and $I^* = (z^{16},y^{10} - x^7z^3)$. 
\end{exm}

Finally, consider the case $r_{12} = r_{32} = 0$ and $c_2 > r_{21} + r_{23}$. From the previous cases, we may assume that $r_{21}, r_{23} > 0$. If $r_{23} \ge c_3$, then we can replace $r_{21}$ by $r_{21} + c_1$ and $r_{23}$ by $r_{23} - c_3$. Thus, we may assume that $0 < r_{23} < c_3$. Then, the lattice $\Lambda$ is generated by $\w_1 = (c_1,0,-c_3)$ and $\w_2 = (-r_{21}, c_2, -r_{23})$. We fix the following notation. Let $\delta_1 = c_1 - c_3$ and $\delta_2 = c_2 - r_{21} - r_{23}$. Since $\w = (-(b-a)/d,b/d,-a/d) \in \Lambda$, there exist $\alpha$ and $\beta$ such that $\w = \alpha \w_1 + \beta \w_2$. In particular, $\beta c_2 = b/d$ and $\alpha \delta_1 + \beta \delta_2 = 0$. Denote $\alpha_h = -\alpha$ and $\beta_h = \beta$. Then we have 
\begin{equation}\label{eq_ci_5_1}
    \begin{cases}
    \alpha_h c_1 + \beta_h r_{21} &= (b-a)/d\\
    \beta_h c_2 &= b/d \\
    -\alpha_h c_3 + \beta_h r_{23} & = a/d.
    \end{cases}
\end{equation}
First, we claim

\begin{lem}\label{lem_ci_5_1} Assume that $r_{12} = 0$ and $c_2 > r_{21} + r_{23}$. Then $\frac{\delta_2}{\delta_1} < \frac{r_{23}}{c_{3}}$.   
\end{lem}
\begin{proof} By \eqref{eq_ci_5_1}, $\beta_h r_{23} - \alpha_h c_3 = a/d > 0$. Hence,
$$\frac{\delta_2}{\delta_1} = \frac{\alpha_h}{\beta_h} < \frac{r_{23}}{c_3}.$$ 
The conclusion follows.    
\end{proof}
Let $\frac{\eta}{\epsilon} = M(\frac{\delta_2}{\delta_1}, \frac{r_{23}}{r_3})$, i.e., $\eta$ and $\epsilon$ are the smallest positive integers such that $\frac{\delta_2}{\delta_1} < \frac{\eta}{\epsilon} \le \frac{r_{23}}{r_3}.$ We now define the following two sets.

Let $\alpha_0 = 1, \beta_0 =0$. For each $i \ge 0$, consider the system of inequalities 
\begin{equation} \label{eq_ci_5_2}
\begin{cases}
     \alpha c_3 - \beta r_{23} & > 0 \\
    \alpha c_{3} - \beta r_{23} & < \alpha_i c_{3} - \beta_i r_{23} \\
    \beta &> \beta_i.
\end{cases}
\end{equation}
When the system \eqref{eq_ci_5_2} has solutions, let $\beta_{i+1}$ be the smallest solution among $\beta$. Then $\alpha_{i+1} = \varphi_r(\beta_{i+1})$ where $r = r_{23}/c_3$. There can be only finitely many such $\beta_i$ as the sequence $\{ \alpha_i c_3 - \beta_i r_{23} \}$ is a strictly decreasing sequence of positive integers. Let $t$ be the largest index of $i$ such that $\beta_t < \epsilon$ and $A = \{(\alpha_i,\beta_i) \mid i = 0, \ldots, t\}$. Since $\alpha_0 c_3 -\beta_0 r_{23} = c_3$, $|A| \le c_3$. For each $\alpha,\beta \in \NN$ such that $\alpha c_3 - \beta r_{23} \ge 0$, let $p_{\alpha,\beta} = y^{\beta c_2} z^{\alpha c_3 - \beta r_{23}}$. We have 

\begin{lem}\label{lem_ci_5_2} Assume that $\alpha' > 0$ and $\beta' \ge 0$ be such that $\alpha' c_{3} - \beta' r_{23} \ge 0$. Let $g = y^{\epsilon c_2}$. Then $p_{\alpha',\beta'} \in L = (g,p_{\alpha,\beta} \mid (\alpha,\beta) \in A)$.    
\end{lem}
\begin{proof}
    The proof is similar to that of Lemma \ref{lem_ci_2_2}.
\end{proof}

Let $\gamma_0 = 0$, $\sigma_0 = 1$. For each $i \ge 0$, consider the system of inequalities 
\begin{equation} \label{eq_ci_5_3}
\begin{cases}
     -\gamma \delta_1 + \sigma \delta_2 & > 0 \\
    -\gamma c_{3} + \sigma r_{23} & < \gamma_i c_{3} + \sigma_i r_{23} \\
    \gamma &> \gamma_i.
\end{cases}
\end{equation}
When the system \eqref{eq_ci_5_3} has solutions, let $\gamma_{i+1}$ be the smallest solution among $\gamma$. Then $\sigma_{i+1} = \varphi_r ( \gamma_{i+1})$ where $r = \delta_1 / \delta_2$. There can be only finitely many such $\gamma_i$ as the sequence $\{ -\gamma_i c_3 + \sigma_i r_{23} \}$ is a strictly decreasing sequence of positive integers. Let $u$ be the largest index of $i$ and $B = \{(\gamma_i,\sigma_i) \mid i = 0, \ldots, u\}$. Since $-\gamma_0 c_3 + \sigma_0 r_{23} = r_{23}$, $|B| \le r_{23}$. For each $\gamma,\sigma \ge 0$ such that $ -\gamma \delta_1 + \sigma \delta_2 > 0$, let $q_{\gamma,\delta} = x^{\gamma c_1 + \sigma r_{21}} z^{-\gamma c_3 + \sigma r_{23}}$. We have 

\begin{lem}\label{lem_ci_5_3} Assume that $\gamma',\sigma'$ be such that $-\gamma' \delta_1 + \sigma' \delta_2 > 0$. Then $q_{\gamma',\sigma'} \in L = (q_{\gamma,\sigma} \mid (\gamma,\sigma) \in B)$.     
\end{lem}
\begin{proof}
    The proof is similar to that of Lemma \ref{lem_ci_2_3}.
\end{proof}

\begin{lem}\label{lem_ci_5} Assume that $r_{12} = 0$. Let $r_{21}, r_{23}$ with $0 < r_{23} < c_3$ be such that $c_2 n_2 = r_{21} n_1 + r_{23} n_3$. Assume that $c_2 > r_{21} + r_{23}$. Let $g = y^{\epsilon c_2}$ and $h = y^{b/d} - x^{(b-a)/d}z^{a/d}$. Then
\begin{equation}
    I^* = (g,h) + (p_{\alpha,\beta} \mid (\alpha,\beta) \in A) + (q_{\gamma,\sigma} \mid (\gamma,\sigma) \in B).
\end{equation}
Furthermore, 
\begin{enumerate}
    \item If $\beta_h > \epsilon$, then $I^* = (g) + (p_{\alpha,\beta} \mid (\alpha,\beta) \in A) + (q_{\gamma,\sigma}\mid (\gamma,\sigma) \in B)$.
    \item If $\beta_h \le \epsilon$ then $I^* = (h) + (p_{\alpha,\beta} \mid (\alpha,\beta) \in A, \beta < \beta_h) + (q_{\gamma,\sigma} \mid (\gamma,\sigma) \in B)$.     
    \item The minimal generators of $I^*$ form a Gr\"obner basis of $I^*$ with respect to the reverse lexicographic order.
    \item $\mu(I^*) \le c_3 + 2$.
\end{enumerate}
\end{lem}
\begin{proof}The lattice $\Lambda$ is generated by $\w_1 = (c_1,0,-c_3)$ and $\w_2 = (-r_{21}, c_2, -r_{23})$. Let 
$$L = (g,h) + (p_{\alpha,\beta} \mid (\alpha,\beta) \in A) + (q_{\gamma,\sigma} \mid (\gamma,\sigma) \in B).$$
For clarity, we divide the proof into several steps. The first two steps can be proved in similar manners as those of Lemma \ref{lem_ci_2}.

\noindent \textbf{Step 1.} $L \subseteq I^*$. 

\vspace{2mm}

\noindent \textbf{Step 2.} $I^* \subseteq L$.

\vspace{2mm}

\noindent \textbf{Step 3.} Assume that $\beta_h > \epsilon$. Then $h \in L = (g) + (p_{\alpha,\beta} \mid (\alpha,\beta) \in A) + (q_{\gamma,\sigma} \mid (\gamma,\sigma) \in B).$

Since $\beta_h > \epsilon$ and $g = y^{\epsilon c_2}$, it suffices to prove that $q_{\alpha_h,\beta_h} \in L$. Let $\beta_h = k \epsilon + \beta$ for $0 \le \beta < \beta_h$. If $\beta = 0$, $k \ge 2$, let $\gamma = \alpha_h - (k-1) \eta$ and $\sigma = \beta_h - (k-1) \epsilon$. If $\beta > 0$, let $\gamma = \alpha_h - k \eta$ and $\sigma = \beta_h - k \epsilon$. Then we have $\frac{\alpha_h}{\beta_h} > \frac{\gamma}{\sigma}$. Furthermore, 
$$ (-\alpha_h c_3 + \beta_h r_{23}) - (-\gamma c_3 + \sigma r_{23}) = k (-\eta c_3 + \epsilon r_{23}) \ge 0$$
Hence, $q_{\alpha,\beta}$ is divisible by $q_{\gamma,\sigma}$ which belongs to $L$ by Lemma \ref{lem_ci_5_3}.

\vspace{2mm}

\noindent \textbf{Step 4.} Assume that $\beta_h \le \epsilon$. Let $L = (h) + (p_{\alpha,\beta} \mid (\alpha,\beta) \in A, \beta < \beta_h) + (q_{\gamma,\sigma} \mid (\gamma,\sigma) \in B)$. Then $g$ and $p_{\alpha,\beta}$ with $(\alpha,\beta) \in A$ and $\beta \ge \beta_h$ belongs to $M$.

For $g$: Let $\epsilon = k \beta_h + \beta$ where $0 \le \beta < \beta_h$. If $\beta = 0$, then we have $\eta = k \alpha_h + 1$. Furthermore, $\frac{\eta}{\epsilon} \le \frac{r_{23}}{c_3}$. Hence, $ k \beta_h r_{23} \ge (k \alpha_h + 1) c_3$. Thus 
$$g - f_{-k \alpha_h \w_1 + k \beta_h \w_2} = x^{k \alpha_h c_1 + k \beta_h r_{21}} z^{-k \alpha_h c_3 + k \beta_h r_{23}} \in (z^{c_3}).$$
Now assume that $\beta > 0$. Then we also have $\eta \ge k \alpha_h + 1$. Let $\alpha = \eta - k \alpha_h$ and $\beta = \epsilon - k \beta_h$. Then, $\alpha/\beta > \alpha_h/\beta_h$. By definition of $\epsilon$ we must have $\alpha/\beta > r_{23}/c_3$. Now, we have 
$$g - y^{\beta c_2} f_{-k \alpha_h \w_1 + k \beta_h \w_2} = x^{k\alpha_h c_1 + k \beta_h r_{21}} y^{\beta c_2} z^{-k \alpha_h c_3 + k \beta_h r_{23}}.$$
Since $(-k \alpha_h c_3 + k \beta_h r_{23}) - (\alpha c_3 - \beta r_{23}) = -\eta c_3 + \epsilon r_{23} \ge 0$, $g - y^{\beta c_2} f_{-k \alpha_h \w_1 + k \beta_h \w_2}$ is divisible by $p_{\alpha,\beta}$ which belongs to $L$ by Lemma \ref{lem_ci_5_2}.

For $p_{\alpha,\beta}$ with $(\alpha,\beta) \in A$ and $\beta \ge \beta_h$. As before, we write $\beta = k \beta_h + \beta'$. Then we must have $\alpha \ge k \beta_h + 1$. Let $\alpha = k \alpha_h + \alpha'$. If $\beta' = 0$, we deduce that $p_{\alpha,\beta} - f_{-k \alpha_h\w_1 + k \beta_h \w_2} \in M$ as in the previous case. If $\beta' > 0$, then we also have $\frac{\alpha'}{\beta'} > \frac{r_{23}}{c_3}$. Hence, $p_{\alpha,\beta} - f_{-k \alpha_h \w_1 + k \beta_h \w_2} \in (p_{\alpha',\beta'})$.

\vspace{2mm}

\noindent \textbf{Step 5.} Assume that $\beta_h \le \epsilon$. Then $\{ h \} \cup \{p_{\alpha,\beta} \mid (\alpha,\beta) \in A, \beta < \beta_h \} \cup \{q_{\gamma,\sigma} \mid (\gamma,\sigma) \in B\}$ form a Gr\"obner basis of $I^*$ with respect to the revlex order. 

Since $\operatorname{in}(h) = y^{\beta_h c_2}$, it suffices to prove that for any $(\alpha,\beta) \in A, \beta < \beta_h$, the S-pair $S(h,p_{\alpha,\beta})$ reduces to $0$. We have 
$$S(h,p_{\alpha,\beta}) = x^{\alpha_h c_1 + \beta_h r_{21}} z^{(\alpha - \alpha_h) c_3 + (\beta_h - \beta) r_{23}}.$$
If $\alpha \ge \alpha_h$ then we have the exponent of $z$ is at least $r_{23}$, hence $S(h,p_{\alpha,\beta}) \in (q_{0,1})$. Thus, we may assume that $\alpha < \alpha_h$. Since $\frac{\delta_2}{\delta_1} < \frac{r_{23}}{c_3} \le \frac{\alpha}{\beta}$, we deduce that $\frac{\alpha_h - \alpha}{\beta_h - \beta} < \frac{\delta_2}{\delta_1}$. Hence $S(h,p_{\alpha,\beta})$ is divisible by $q_{\alpha_h - \alpha, \beta_h - \beta}$ which belongs to $L = (q_{\gamma,\sigma} \mid (\gamma,\sigma) \in B)$. The conclusion follows.

\vspace{2mm}

\noindent \textbf{Step 6.} $\mu(I^*) \le c_3 + 2$. Let $c_3 = k r_{23} + s$ with $0 \le s < r_{23}$. If $s = 0$, then $-\gamma c_3 + \sigma r_{23}$ is divisible by $r_{23}$. Hence the system \eqref{eq_ci_5_3} has no solution for $i \ge 0$. Hence $|B| = 1$. Thus, $\mu(I^*) \le 1 + 1 + |A| \le c_3 + 2$.

Now assume that $s > 0$. Then $(1,i) \in A$ for $i = 0, \ldots, k$ and $p_{1,k} = y^{kc_2}z^s$. Thus $|A| \le k + s$. Furthermore, $|B| \le r_{23}$. Thus, 
$$\mu(I^*) \le 1 + k + s + r_{23} \le c_3 + 2.$$
That concludes the proof of the lemma.
\end{proof}

\begin{exm}
   Let H =  $\langle 480, 503, 1950 \rangle$. Then $I = (y^{30} - x^3z^7,x^{65} - z^{16})$ and $I^* = (x^3z^7,z^{16},y^{30}z^9,y^{60}z^2,x^{74}z^5,x^{145}z^3,y^{210}).$
\end{exm}

\begin{exm}
   Let $H = \langle 160,169,460 \rangle$. Then $I = (y^{20} - xz^7,x^{23} - z^8)$ and $I^* = (xz^7,z^8,y^{20}z,x^{25}z^6,x^{49}z^5,x^{73}z^4,y^{100} - x^{97}z^3)$.
\end{exm}

\section{Non complete intersection monomial space curves}\label{sec_non_ci}

We keep the notation as in Section \ref{sec_pre}. In this section, we consider the case where $I_H$ is not a complete intersection. In particular, $r_{ij} \neq 0$ for all $i,j$. Then, the lattice $\Lambda$ is generated by $\w_1 = (c_1,-r_{12},-r_{13})$ and $\w_2 = (-r_{21},c_2,-r_{23})$. 

\begin{lem}\label{lem_nci_1}
    Assume that $r_{ij} \neq 0$ and $c_2 < r_{21} + r_{23}$. Then $I^* = (y^{r_{12}}z^{r_{13}}, y^{c_2},z^{c_3})$.
\end{lem}
\begin{proof} Let $J = (y^{r_{12}} z^{r_{13}}, y^{c_2},z^{c_3})$. Then $J \subseteq I^*$. We need to show that for any $\w \in \Lambda$, $f_\w^* \in J$. By Lemma \ref{lem_tangent_cone_space_curve}, there are two cases.

\noindent \textbf{Case 1.} $\w = \alpha \w_1 + \beta \w_2 = (\alpha c_1 -\beta r_{21}, -\alpha r_{12} + \beta c_2, -\alpha r_{13} -\beta r_{23})$. There are two subcases. 

Case 1.a. $\alpha c_1 - \beta r_{21} > 0$. Then $\w$ is of type $1$ or type $3$. In either case, we have $f_\w^* \in (z^{\alpha r_{13} + \beta r_{23}}) \in J$, as $c_3 = r_{13} + r_{23}$. 

Case 1.b. $\alpha c_1 - \beta r_{21} < 0$. Then $\alpha < \beta$. Hence, $-\alpha r_{12} + \beta c_2 > c_2$. Therefore, $y^{-\alpha r_{12} + \beta c_2} \in (y^{c_2})$ and $x^{\beta r_{21} - \alpha c_1} z^{\alpha r_{13} + \beta r_{23}} \in J$. Thus $f_\w^* \in J$.

\noindent \textbf{Case 2.} $\w = \alpha \w_1 - \beta \w_2 = (\alpha c_1 + \beta r_{21}, -\alpha r_{12} - \beta c_2, -\alpha r_{13} + \beta r_{23})$. There are two subcases. 

Case 2.a. $-\alpha r_{13} + \beta r_{23} < 0$. Then $f_\w^* = y^{\alpha r_{12} + \beta c_2} z^{\alpha r_{13} - \beta r_{23}} \in (y^{c_2})$. 

Case 2.b. $-\alpha r_{13} + \beta r_{23} > 0$. Then $f$ is of type $2$. Since $c_2 < r_{21} + r_{23}$, we have 
$$  \alpha c_1 + \beta r_{21} - \alpha r_{13} + \beta r_{23} > \alpha r_{12} + \beta c_2. $$
Hence $f_\w^* = y^{\alpha r_{12} + \beta c_2} \in J$. The conclusion follows. 
\end{proof}
\begin{exm}Let $H = \langle 13,20,31\rangle$. Then $I = (x^7 - y^3 z, y^7 - x^6 z^2, z^3 - xy^4)$ and $I^* = (z^3,y^3z,y^7)$.
\end{exm}

\begin{lem}\label{lem_nci_addition} Assume that $r_{ij} \neq 0$ and $c_2 \ge r_{21} + r_{23}$. Let $\alpha,\beta > 0$ and $\w = \alpha \w_1 + \beta \w_2$. Then $f_{\w}^* \in (z^{c_3})$.
\end{lem}
\begin{proof} We have $\w = \alpha \w_1 + \beta \w_2 = (\alpha c_1 - \beta r_{21}, -\alpha r_{12} + \beta c_2, -\alpha r_{13} - \beta r_{23})$. There are two cases. 

\noindent \textbf{Case 1.}  $\alpha c_1 - \beta r_{21} \ge 0$. Then $\w$ is of type $1$ or type $3$. In either case, we have $f_\w^* \in (z^{\alpha r_{13} + \beta r_{23}}) \subseteq (z^{c_3})$, as $c_3 = r_{13} + r_{23}$. 

\noindent \textbf{Case 2.}  $\alpha c_1 - \beta r_{21} < 0$. We have 
$$\beta r_{21} - \alpha c_1 + \alpha r_{13} + \beta r_{23} < -\alpha r_{12} + \beta c_2.$$
Hence, $f_\w^* = x^{\beta r_{21} - \alpha c_1} z^{\alpha r_{13} + \beta r_{23}} \in (z^{c_3})$.

The conclusion follows.
\end{proof}

\begin{lem}\label{lem_nci_2} Assume that $r_{ij} \neq 0$ and $c_2 = r_{21} + r_{23}$. Then $I^* = (y^{r_{12}} z^{r_{13}}, y^{c_2} - x^{r_{21}} z^{r_{23}},z^{c_3})$.  
\end{lem}
\begin{proof}
    Let $J = (y^{r_{12}} z^{r_{13}}, y^{c_2} - x^{r_{21}} z^{r_{23}},z^{c_3})$. Then $J \subseteq I^*$. We need to show that for any $\w \in \Lambda$, $f_\w^* \in J$. By Lemma \ref{lem_tangent_cone_space_curve} and Lemma \ref{lem_nci_addition}, we may assume that $\w = \alpha \w_1 - \beta \w_2 = (\alpha c_1 + \beta r_{21}, -\alpha r_{12} - \beta c_2, -\alpha r_{13} + \beta r_{23})$. There are two cases. 

\noindent \textbf{Case 1.}  $-\alpha r_{13} + \beta r_{23} \le 0$. Then $f_\w* = y^{\alpha r_{12} + \beta c_2} z^{\alpha r_{13} - \beta r_{23}}$. Now we have 
$$f_\w^* - y^{\alpha r_{12}} f_{\beta \w_2} = x^{\beta r_{21}} z^{\beta r_{23}} z^{\alpha r_{13} - \beta r_{23}} = (y^{r_{12}} z^{r_{13}})^\alpha x^{\beta r_{21}} \in J.$$
    
\noindent \textbf{Case 2.}  $-\alpha r_{13} + \beta r_{23} > 0$. Then $f_\w$ is of type $2$. But by assumption, we have 
$$\alpha c_1 + \beta r_{21} - \alpha r_{13} + \beta r_{23} > \alpha r_{12} + \beta c_2.$$
Hence $f_\w^* = y^{\alpha r_{12} + \beta c_2}$. Now we have 
$$    f_\w^* - y^{\alpha r_{12}} f_{\beta \w_2} = y^{\alpha r_{12}} x^{\beta r_{21}} z^{\beta r_{23}}.$$
Since $-\alpha r_{13} + \beta r_{23} > 0$, $y^{\alpha r_{12}} z^{\beta r_{23}}$ is divisible by $(y^{r_{12}} z^{r_{13}})^\alpha$.
\end{proof}

\begin{exm}This is typical case of non-complete intersection monomial curve with large shift. For example, consider the semigroup obtained by shiting the semigroup from the previous example, $H = \langle 193, 200, 211 \rangle$. Then we have $I = (x^{16} - y^7 z^8, y^{18} - x^{11} z^7, z^{15} - x^5 y^{11})$ and $I^* = (y^7 z^8, y^{18} - x^{11} z^7, z^{15})$.   
\end{exm}

Now assume that $c_2 > r_{21} + r_{23}$. Let $\delta_1 = c_1 - r_{12} - r_{13}$ and $\delta_2 = c_2 - r_{21} - r_{23}$. Since $\w = (-(b-a)/d,b/d,-a/d) \in \Lambda$, there exist $\alpha$ and $\beta$ such that $\w = \alpha \w_1 + \beta \w_2$. In particular, $\alpha \delta_1 + \beta \delta_2 = 0$ and 
$\alpha c_1 - \beta r_{21} = -(b-a)/d$ and $-\alpha r_{12} + \beta c_2 = b/d$. Thus, we must have $\alpha < 0$ and $\beta > 0$. We fix $\alpha_h = -\alpha$ and $\beta_h = \beta$. Then we have

\begin{equation}\label{eq_nci_3_1}
    \begin{cases}
    \alpha_h c_1 + \beta_h r_{21} &= (b-a)/d\\
    \alpha_h r_{12} + \beta_h c_2 &= b/d \\
    -\alpha_h r_{13} + \beta_h r_{23} & = a/d.
    \end{cases}
\end{equation}
First, we claim

\begin{lem}\label{lem_nci_3_1} Assume that $r_{ij} \neq 0$ for all $i,j$ and $c_2 > r_{21} + r_{23}$. Then $\frac{\delta_2}{\delta_1} < \frac{r_{23}}{r_{13}}$.    
\end{lem}
\begin{proof}
    Since $-\alpha_h r_{13} + \beta_h r_{23} = a/d > 0$, we have 
    $$\frac{\alpha_h}{\beta_h} = \frac{\delta_2}{\delta_1} < \frac{r_{23}}{r_{13}}.$$
    The conclusion follows.
\end{proof}

Let $\frac{\eta}{\epsilon} = M(\frac{\delta_2}{\delta_1},\frac{r_{23}}{r_{13}})$, i.e., $\eta$ and $\epsilon$ are the smallest positive integers such that $\frac{\delta_2}{\delta_1} < \frac{\eta}{\epsilon} \le \frac{r_{23}}{r_{13}}$. We now define the following two sets.

Let $\alpha_0 = 1, \beta_0 =0$. For each $i \ge 0$, consider the system of inequalities 
\begin{equation} \label{eq_nci_3_2}
\begin{cases}
     \alpha r_{13} - \beta r_{23} & > 0 \\
    \alpha r_{13} - \beta r_{23} & < \alpha_i r_{13} - \beta_i r_{23} \\
    \beta &> \beta_i.
\end{cases}
\end{equation}
When the system \eqref{eq_nci_3_2} has solutions, let $\beta_{i+1}$ be the smallest solution among $\beta$. Then $\alpha_{i+1} = \varphi_r(\beta_{i+1})$ where $r = r_{23}/r_{13}$. There can be only finitely many such $\beta_i$ as the sequence $\{ \alpha_i r_{13} - \beta r_{23} \}$ is a strictly decreasing sequence of positive integers. Let $t$ be the largest index of $i$ such that $\beta_t < \epsilon$ and $A = \{(\alpha_i,\beta_i) \mid i = 0, \ldots, t\}$. Since $\alpha_0 r_{13} - \beta_0 r_{23} = r_{13}$, $|A| \le r_{13}$. For each $\alpha,\beta \in \NN$ such that $\alpha r_{13} - \beta r_{23} \ge 0$, let $p_{\alpha,\beta} = y^{\alpha r_{12} + \beta c_2} z^{\alpha r_{13} - \beta r_{23}}$. We have
\begin{lem}\label{lem_nci_aux_1} Assume that $\alpha' > 0$, $\beta' \ge 0$ be such that $\alpha' r_{13} - \beta' r_{23} \ge 0$. Then $p_{\alpha',\beta'} \in L = (g,p_{\alpha,\beta} \mid (\alpha,\beta) \in A)$.    
\end{lem}
\begin{proof} If $\beta' = 0$, then $p_{\alpha',\beta'}$ is divisible by $p_{1,0}$. Assume that $\beta' \ge \epsilon$. Since $\frac{\alpha'}{\beta'} \ge \frac{r_{23}}{r_{13}} > \frac{\eta}{\epsilon}$, $\alpha' > \eta$. Thus, $p_{\alpha',\beta'}$ is divisible by $g$.

Thus, we may assume that $\beta' > 0$ and $\beta' < \epsilon$. Let $i$ be the largest index such that $\beta_i < \beta'$. Since $\frac{\alpha'}{\beta'} \ge r$ and $\alpha_i = \varphi_r(\beta_i)$ where $r = r_{23}/ r_{13}$, we deduce that $\alpha' \ge \alpha_i$. If $\alpha' r_{13} - \beta r_{23} \ge \alpha_i r_{13} - \beta_i r_{23}$ then $p_{\alpha',\beta'}$ is divisible by $p_{\alpha_i,\beta_i}$. If $\alpha' r_{13} - \beta' r_{23} < \alpha_i r_{13} - \beta_i r_{23}$, then $(\alpha',\beta')$ is a solution of the system \eqref{eq_nci_3_2}. Thus, we must have $\beta' = \beta_{i+1}$ as we assume that $\beta_{i+1} \ge \beta'$. Hence, $\alpha' \ge \alpha_{i+1}$. Thus, $p_{\alpha',\beta'}$ is divisible by $p_{\alpha_{i+1},\beta_{i+1}}$. The conclusion follows.    
\end{proof}

Let $\gamma_0 = 0$, $\sigma_0 = 1$. For each $i \ge 0$, consider the system of inequalities 
\begin{equation} \label{eq_ci_8}
\begin{cases}
     -\gamma \delta_1 + \sigma \delta_2 & > 0 \\
    -\gamma r_{13} + \sigma r_{23} & < \gamma_i r_{13} + \sigma_i r_{23} \\
    \gamma &> \gamma_i.
\end{cases}
\end{equation}
When the system \eqref{eq_ci_8} has solutions, let $\gamma_{i+1}$ be the smallest solution among $\gamma$. Then $\sigma_{i+1} = \varphi_r ( \gamma_{i+1})$ where $r = \delta_1 / \delta_2$. There can be only finitely many such $\gamma_i$ as the sequence $\{ -\gamma_i r_{13} + \sigma_i r_{23} \}$ is a strictly decreasing sequence of positive integers. Let $u$ be the largest index of $i$ and $B = \{(\gamma_i,\sigma_i) \mid i = 0, \ldots, u\}$. Since $-\gamma_0 r_{13} + \sigma_0 r_{23} = r_{23}$, $|B| \le r_{23}$. For each $\gamma,\sigma \in \NN$ such that $-\gamma r_{13} + \sigma r_{23} > 0$, let $q_{\gamma,\delta} = x^{\gamma c_1 + \sigma r_{21}} z^{-\gamma r_{13} + \sigma r_{23}}$. We have 
\begin{lem}\label{lem_nci_aux_2} Assume that $\sigma' > 0$ and $\gamma' \ge 0$ be such that $-\gamma' \delta_1 + \sigma' \delta_2 > 0$. Then $q_{\gamma',\sigma'} \in L = (q_{\gamma,\sigma} \mid (\gamma,\sigma) \in B).$    
\end{lem}
\begin{proof} If $\gamma' = 0$ then $p_{\gamma',\sigma'}$ is divisible by $q_{0,1}$. Thus, we may assume that $\gamma' > 0$. Let $i$ be the largest index such that $\gamma_i < \gamma'$, i.e., $\gamma_{i+1} \ge \gamma' > \gamma_i$. Since $\sigma_i = \varphi_r(\gamma_i)$ and $\frac{\sigma'}{\gamma'} > r$ where $r = \delta_1/\delta_2$, we deduce that $\sigma' \ge \sigma_i$. If $-\gamma' r_{13} + \sigma' r_{23} \ge \gamma_i r_{13} + \sigma_i r_{23} $ then $q_{\gamma',\sigma'}$ is divisible by $q_{\gamma_i,\sigma_i}$. If $-\gamma' r_{13} + \sigma' r_{23} < \gamma_i r_{13} + \sigma_i r_{23}$, then $(\gamma',\sigma')$ is a solution of \eqref{eq_ci_8}. Thus, $\gamma' \ge \gamma_{i+1}$. By our definition, we must have $\gamma' = \gamma_{i+1}$ and $\sigma' \ge \sigma_{i+1}$. Hence, $q_{\gamma',\sigma'}$ is divisible by $q_{\gamma_{i+1},\sigma_{i+1}}$. The conclusion follows.    
\end{proof}

\begin{lem}\label{lem_nci_3} Assume that $r_{ij} \neq 0$ and $c_2 > r_{21} + r_{23}$. Let $g = y^{\eta r_{12} + \epsilon c_2}$ and $h = y^{b/d} - x^{(b-a)/d}z^{a/d}$. Then
$$I^* = (z^{c_3},g,h) + (p_{\alpha,\beta} \mid (\alpha,\beta) \in A) + (q_{\gamma,\sigma} \mid (\gamma,\sigma) \in B).$$
Furthermore, 
\begin{enumerate}
    \item If $\beta_h < \epsilon$ then $I^* = (h) + (p_{\alpha,\beta} \mid (\alpha,\beta) \in A, \beta < \beta_h) + (q_{\gamma,\sigma} \mid (\gamma,\sigma) \in B).$
    \item If $\beta \ge \epsilon$ then $I^* = (g) + (p_{\alpha,\beta} \mid (\alpha,\beta) \in A) + (q_{\gamma,\sigma} \mid (\gamma,\sigma) \in B)$.
    \item The minimal generators of $I^*$ form a Gr\"obner basis of $I^*$ with respect to the reverse lexicographic order.
    \item $\mu(I^*) \le \max(r_{13},r_{23}) + 3$.
\end{enumerate}
\end{lem}
\begin{proof}
    Let $$L = (g,h) + (p_{\alpha,\beta} \mid (\alpha,\beta) \in A) + (q_{\gamma,\sigma} \mid (\gamma,\sigma) \in B.$$
For clarity, we divide the proof into several steps. The first two steps can be proved in similar manners as those of Lemma \ref{lem_ci_2}. We recall some arguments for completeness.

\vspace{2mm}

\noindent \textbf{Step 1.} $L \subseteq I^*$. 

\vspace{2mm}

\noindent \textbf{Step 2.} $I^* \subseteq L$. By Lemma \ref{lem_tangent_cone_space_curve} and Lemma \ref{lem_nci_addition}, we may assume that $\w = \alpha \w_1 - \beta \w_2 = (\alpha c_1 + \beta r_{21}, -\alpha r_{12} - \beta c_2, -\alpha r_{13} + \beta r_{23})$. There are two cases.

\noindent \textbf{Case 1.} $-\alpha r_{13} + \beta r_{23} \le 0$. Then, $f_\w^* = p_{\alpha,\beta} \in (g,p_{\alpha,\beta} \mid (\alpha,\beta) \in A)$ by Lemma \ref{lem_nci_aux_1}. 

\noindent \textbf{Case 2.} $-\alpha r_{13} + \beta r_{23} > 0$. There are three subcases.

Case 2.a. $\alpha \delta_1 < \beta \delta_2$. Then $f_\w^* = q_{\alpha,\beta} \in (q_{\gamma,\sigma} \mid (\gamma,\sigma) \in B)$ by Lemma \ref{lem_nci_aux_2}. 

Case 2.b. $\alpha \delta_1 = \beta \delta_2$, then there must exist $k$ such that $\alpha = k \alpha_h$ and $\beta = k \beta_h$. Hence, $f_\w^* = f_\w \in (h)$. 

Case 2.c. $\alpha \delta_1 > \beta \delta_2$. Then $f_\w^*  = y^{\alpha r_{12} + \beta c_2}$. Furthermore, we have $\frac{\delta_2}{\delta_1} < \frac{\alpha}{\beta} < \frac{r_{23}}{r_{13}}$. Hence, by the definition of $\eta$ and $\epsilon$ we must have $\beta \ge \epsilon$ and $\alpha \ge \eta$. Hence, $f_\w^* \in (g)$.

\vspace{2mm}

\noindent \textbf{Step 3.} Assume that $\beta_h > \epsilon$. Then $h \in M = (g) + (p_{\alpha,\beta} \mid (\alpha,\beta) \in A) + (q_{\gamma,\sigma} \mid (\gamma,\sigma) \in B).$

Since $\eta = \varphi_r(\epsilon)$ where $r = \delta_2/\delta_1 = \alpha_h/\beta_h$, we deduce that $\alpha_h \ge \eta$. Thus, $y^{\alpha_h r_{12} + \beta_h c_2} \in (g)$. It suffices to prove that $q_{\alpha_h,\beta_h} \in M$. Let $\beta_h = k \epsilon + \beta$ for $0 \le \beta < \beta_h$. If $\beta = 0$, then $k \ge 2$, let $\gamma = \alpha_h - (k-1) \eta$ and $\sigma = \beta_h - (k-1) \epsilon$. If $\beta > 0$, let $\gamma = \alpha_h - k \eta$ and $\sigma = \beta_h - k \epsilon$. Then we have $\frac{\delta_2}{\delta_1} = \frac{\alpha_h}{\beta_h} > \frac{\gamma}{\sigma}$. Hence, $q_{\alpha,\beta}$ is divisible by $q_{\gamma,\sigma}$. By Lemma \ref{lem_nci_aux_2}, $q_{\gamma,\sigma} \in M$. The conclusion follows.

\vspace{2mm}

\noindent \textbf{Step 4.} Assume that $\beta_h \le \epsilon$. Let $M = (h) + (p_{\alpha,\beta} \mid (\alpha,\beta) \in A, \beta < \beta_h) + (q_{\gamma,\sigma} \mid (\gamma,\sigma) \in B)$. Then $g$ and $p_{\alpha,\beta}$ with $(\alpha,\beta) \in A$ and $\beta \ge \beta_h$ belongs to $M$.

For $g$: Let $\epsilon = k \beta_h + \beta$ where $0 \le \beta < \beta_h$. If $\beta = 0$, then we have $\eta = k \alpha_h + 1$. Furthermore, $\frac{\eta}{\epsilon} \le \frac{r_{23}}{r_{13}}$. Hence, $ k \beta_h r_{23} \ge (k \alpha_h + 1) r_{13}$. Thus 
$$g - f_{-k \alpha_h \w_1 + k \beta_h \w_2} = x^{k \alpha_h c_1 + k \beta_h r_{21}} z^{-k \alpha_h r_{13} + k \beta_h r_{23}} \in (x^{r_{21}}z^{r_{23}}).$$
Now assume that $\beta > 0$. Then we also have $\eta \ge k \alpha_h + 1$. Furthermore, $\frac{\eta - k\alpha_h}{\epsilon - k \beta_h} > \frac{\alpha_h}{\beta_h}$. By definition of $\epsilon$ we must have $\frac{\eta - k \alpha_h}{\epsilon - k \beta_h} > \frac{r_{23}}{c_3}$. Let $\alpha = \eta - k \alpha_h$. Then, we have 
$$g - y^{\beta c_2} f_{-k \alpha_h \w_1 + k \beta_h \w_2} = x^{k\alpha_h c_1 + k \beta_h r_{21}} y^{\beta c_2} z^{-k \alpha_h c_3 + k \beta_h r_{23}} \in (p_{\alpha,\beta}) \subseteq M,$$
by Lemma \ref{lem_nci_aux_1}.

For $p_{\alpha,\beta}$ with $(\alpha,\beta) \in A$ and $\beta \ge \beta_h$. As before, we write $\beta = k \beta_h + \beta'$. Then we must have $\alpha \ge k \beta_h + 1$. Let $\alpha = k \alpha_h + \alpha'$. If $\beta' = 0$, we deduce that $p_{\alpha,\beta} - f_{-k \alpha_h\w_1 + k \beta_h \w_2} \in M$ as in the previous case. If $\beta' > 0$, then we also have $\frac{\alpha'}{\beta'} > \frac{r_{23}}{c_3}$. Hence, $p_{\alpha,\beta} - f_{-k \alpha_h \w_1 + k \beta_h \w_2} \in (p_{\alpha',\beta'})$.

\vspace{2mm}

\noindent \textbf{Step 5.} Assume that $\beta_h \le \epsilon$. Then $\{ h \} \cup \{p_{\alpha,\beta} \mid (\alpha,\beta) \in A, \beta < \beta_h \} \cup \{q_{\gamma,\sigma} \mid (\gamma,\sigma) \in B\}$ form a Gr\"obner basis of $I^*$ with respect to the revlex order. 

Since $\operatorname{in}(h) = y^{\alpha_h r_{12} + \beta_h c_2}$, it suffices to prove that for any $(\alpha,\beta) \in A, \beta < \beta_h$, the S-pair $S(h,p_{\alpha,\beta})$ reduces to $0$. We have 
$$S(h,p_{\alpha,\beta}) = x^{\alpha_h c_1 + \beta_h r_{21}} z^{(\alpha - \alpha_h) r_{13} + (\beta_h - \beta) r_{23}}.$$
If $\alpha \ge \alpha_h$ then the exponent of $z$ is at least $r_{23}$; hence, $S(h,p_{\alpha,\beta}) \in (q_{0,1})$. Thus, we may assume that $\alpha < \alpha_h$. Since $\frac{\delta_2}{\delta_1} < \frac{r_{23}}{r_{13}} \le \frac{\alpha}{\beta}$, we deduce that $\frac{\alpha_h - \alpha}{\beta_h - \beta} < \frac{\delta_2}{\delta_1}$. Hence $S(h,p_{\alpha,\beta}) \in (q_{\alpha_h - \alpha, \beta_h - \beta})$. By Lemma \ref{lem_nci_aux_2}, the conclusion follows.

\vspace{2mm}

\noindent \textbf{Step 6.} $\mu(I^*) \le \max(r_{13},r_{23}) + 3$. There are two cases.

\noindent \textbf{Case 1.} $r_{23} \ge r_{13}$. Let $r_{23} = k r_{13} + s_{13}$ with $0\le s_{13} < r_{13}$. If $s_{13} = 0$ then the system \eqref{eq_ci_5_2} has no solutions. Thus $|A| = 1$. Hence, $\mu(I^*) \le 2 + 1 + r_{23}$. Now assume that $s_{13} > 0$. If $\frac{\delta_2}{\delta_1} \le 1$ then the system \eqref{eq_ci_5_3} has no solutions. Hence $|B| \le 1$. Thus $\mu(I^*) \le r_{13} +3.$ If $\frac{\delta_2}{\delta_1} > 1$ then $(1,1) \in B$. Since $q_{1,1} = x^{c_1 + r_{21}} z^{r_{23} - r_{13}}$, $|B| \le r_{23} - r_{13} + 1$. Thus $\mu(I^*) \le 2 + r_{23} - r_{13} + r_{13} + 1 = r_{23} + 3$.

\noindent \textbf{Case 2.} $r_{23} < r_{13}$. Then $(1,1)\in A$ and $p_{1,1} = y^{r_{12} + c_2} z^{r_{13} - r_{23}}$. Thus $|A| \le r_{13} - r_{23} + 1$. Hence $\mu(I^*) \le 2 + r_{13} - r_{23} + 1 + r_{23} = r_{13} + 3$.

That concludes the proof of the lemma.
\end{proof}

\begin{exm}Let $H = \langle 265, 280,655 \rangle$. Then $I = (x^{6} - y z^2, y^{22} - xz^{9}, x^5 y^{21} - z^{11} )$ and $I^* = (yz^2,xz^{9}, z^{11},x^7z^7,x^{13}z^5,x^{19}z^3,y^{26} - x^{25}z ).$
\end{exm}

\section{Bounding Betti numbers of monomial space curves}\label{sec_betti_bounds}

In this section, using the description of $I_H^*$ in the previous two sections, we prove our main result. 

First, we recall the following general bound on the Betti numbers of a monomial ideal \cite[Theorem 6.29]{MS}. Assume that $n < r$. The cyclic polytope $C_{n,r}$ is defined to be the convex hull of $r$ distinct points on the moment curve $t \to (t^1,\ldots,t^n)$. We have

\begin{thm}\label{thm_upperbound} The number $\beta_i(I)$ of minimal $i$th syzygies of any monomial
ideal $I$ with $r$ generators in $n$ variables is bounded above by the number
$C_{i,n,r}$ of $i$-dimensional faces of the cyclic $n$-polytope with $r$ vertices. If $i = n - 1$ then we even have $\beta_i(I) \le  C_{n-1,n,r} - 1$.
\end{thm}

\begin{thm}
    Let $H$ be a numerical semigroup generated by $\langle n_1,n_2,n_3\rangle$. Let $s = \wdd(\tilde H) = \min(n_1-1,b)$. Then 
    \begin{enumerate}
        \item $\beta_0(I_H^*) \le s+1,$
        \item $\beta_1(I_H^*) \le 3s - 3,$
        \item $\beta_2(I_H^*) \le 2s - 3.$
    \end{enumerate}
    In particular, $\beta_i(I_H^*) \le \beta_i(I_{\tilde H}^*)$ for all $i$.
\end{thm}
\begin{proof} By \cite{H2}, if $\mu(I^*) \le 3$, $I^*$ is Cohen-Macaulay. Hence, $\beta_1(I^*) \le 2$ and $\beta_2(I_H^*) = 0$. Thus, we may assume that $\mu(I^*) \ge 4$. By Lemmas \ref{lem_ci_1}, \ref{lem_ci_2}, \ref{lem_ci_3}, \ref{lem_ci_4}, \ref{lem_ci_5}, \ref{lem_nci_1}, \ref{lem_nci_2}, \ref{lem_nci_3} it remains to consider the following cases: $r_{21} = 0$, $r_{12} = 0$ and $c_2 > r_{21} + r_{23}$, and $r_{ij} \neq 0$ and $c_2 > r_{21} + r_{23}$. We have $\beta_i(I) \le \beta_i(\operatorname{in}(I))$ for any homogeneous ideal $I$, where $\operatorname{in}(I)$ is the initial ideal of $I$ with respect to any monomial order. By Lemmas \ref{lem_ci_2}, \ref{lem_ci_5}, \ref{lem_nci_3}, Theorem \ref{thm_upperbound}, and the fact that $C_{1,3,r} = 3r - 6$ and $C_{2,3,r} = 2r - 4$ it suffices to prove that $\mu(I^*) \le s+1$ in these cases.

First, note that we always have $c_3 < n_1 - 1$. Since $c_3$ is the smallest positive integer such that $c_3 n_3 \in \langle n_1,n_2 \rangle$. Furthermore, the Frobenius number of $\langle n_1,n_2 \rangle$ is $(n_1-1)(n_2-1)$. Hence,
$c_3 n_3 \le (n_1-1)(n_2-1) + 1$. Thus $c_3 < n_1-1$.

We may now assume that $n_1 > b = n_3 - n_1$.

\noindent \textbf{Case 1.} $r_{21} = 0$. By \eqref{eq_ci_2_1}, $\alpha_h r_{12} + \beta_h c_2 = b/d$. Hence, $c_2 < b/d$. Thus $c_3 < c_2 \le b/d-1$. By Lemma \ref{lem_ci_2}, $\mu(I^*) \le c_3 + 2 \le b$.

\noindent \textbf{Case 2.} $r_{12} = 0$ and $c_2 > r_{21} + r_{23}$. By \eqref{eq_ci_5_1}, $\alpha_h c_1 + \beta_h r_{21} = (b-a)/d$, thus $c_1 < (b-a)/d$. Hence, $c_3 < c_1 \le b-2$. By Lemma \ref{lem_ci_5}, $\mu(I^*) \le c_3 + 2 < b$.

\noindent \textbf{Case 3.} $r_{ij}\neq 0$ for all $i,j$ and $c_2 > r_{21} + r_{23}$. By \eqref{eq_nci_3_1}, we have
    \begin{align*}
        \alpha_h c_1 + \beta_h r_{21} &= (b-a)/d \\
        \alpha_h r_{12} + \beta_h c_2 &= b/d.
    \end{align*}
 In particular, $c_1 < (b-a)/d$ and $c_2 < b/d$. Thus, $r_{13} < c_1-1 \le b-3$ and $r_{23} < c_2-1 \le b-2$. By Lemma \ref{lem_nci_3}, $\mu(I^*) \le \max(r_{13},r_{23}) + 3 \le b$.

 The final statement that $\beta_i(I_H^*) \le \beta_i(I_{\tilde H}^*)$ follows immediately as the formula for $\beta_i(I_{\tilde H}^*)$ is given in \cite[Proposition 2.5]{HS} and \cite[Theorem 4.1]{GSS}.
\end{proof}

\begin{rem} \begin{enumerate}
    \item Herzog and Stamate \cite[Conjecture 2.4]{HS} only conjectured the bound for the number of minimal generators of $I_H^*$. We believe it should hold for all Betti numbers as well. 
    \item In \cite{S}, Stamate proved that when $n_1 \ge k_{a,b}$, $\mu(I_H) = \mu(I_H^*)$, where
$$k_{a,b} = b \max (\frac{a}{d}, \frac{b-a}{d} - 1).$$
This can be deduced from our description of $I_H^*$ as well.
\end{enumerate}
\end{rem}

%-------------------------------------------------------
%-------------------------------------------------------
%-------------------------------------------------------
%-------------------------------------------------------

\end{document}